\documentclass{amsart}[12pt]
\usepackage{amssymb,amsmath}
\usepackage{enumerate}
\usepackage[english]{babel}
\usepackage{color}

\newcommand {\ep} {\varepsilon}

\newcommand {\gm} {\gamma}
\newcommand {\ii} {\infty}
\newcommand {\dt} {\delta}
\newcommand {\al} {\alpha}
\newcommand {\bt} {\beta}
\newcommand {\lb} {\lambda}
\newcommand {\Lb} {\Lambda}

\newcommand {\sm} {\setminus}
\newcommand {\su} {\subset}

\newcommand {\wh} {\widehat}
\newcommand {\pp} {\perp}

\newcommand {\mc} {\mathcal}
\newcommand {\mb} {\mathbf}
\newcommand {\mbb} {\mathbb}

\newtheorem{teo}{Theorem}[section]
\newtheorem{pro}{Proposition}[section]
\newtheorem{cor}{Corollary}[section]

\theoremstyle{definition}
\newtheorem{rem}{Remark}[section]
\newtheorem{df}{Definition}[section]

\newtheorem{ex}{Example}[section]

\title{On individual ergodic theorems \\for semifinite von Neumann algebras}
\begin{document}
\date{June 1, 2020}

\begin{abstract}
It is known that, for a positive Dunford-Schwartz operator in a noncommutative $L^p$-space, $1\leq p<\ii$, or, more generally, in a noncommutative Orlicz space with order continuous norm, the corresponding ergodic averages converge bilaterally almost uniformly. We show that these averages converge almost uniformly in each noncommutative symmetric space $E$ such that $\mu_t(x)\to 0$ as $t\to\ii$ for every $x\in E$, where $\mu_t(x)$ is the non-increasing rearrangement of $x$. Noncommutative Dunford-Schwartz-type multiparameter ergodic theorems are studied. A wide  range of noncommutative symmetric spaces for which Dunford-Schwartz-type individual ergodic theorems hold is outlined.
\end{abstract}

\author{VLADIMIR CHILIN}
\address{National University of Uzbekistan, Tashkent, Uzbekistan}
\email{vladimirchil@gmail.com; chilin@ucd.uz}
\author{SEMYON LITVINOV}
\address{Pennsylvania State University \\ 76 University Drive \\ Hazleton, PA 18202, USA}
\email{snl2@psu.edu}

\maketitle

\section{Introduction}
The study of noncommutative individual ergodic theorems in the space of measurable operators associated with
a semifinite von Neumann algebra $\mc M$ equipped with a faithful normal semifinite trace $\tau$ was initiated by F. J. Yeadon. In \cite{ye}, he derived, as a corollary of a noncommutative maximal ergodic inequality in $L^1(\mc M,\tau)$, an individual ergodic theorem for the so-called positive $L^1-L^\ii$\,-\,contraction acting in the noncommutative  space $L^1(\mc M,\tau)$.

The study of individual ergodic theorems beyond $L^1(\mc M,\tau)$ started much later with another fundamental paper by M. Junge and Q. Xu \cite{jx} where, among other results, individual ergodic theorem was extended to the case where a positive Dunford-Schwartz operator acted in the space $L^p(\mc M,\tau)$, $1<p<\ii$. In \cite{li}, this result, for positive $L^1-L^\ii$-\,contractions, was derived  directly from Yeadon's maximal ergodic inequality with the help of the notion of uniform equicontinuity in measure at zero of a family of linear maps from a Banach space into the space of $\tau$\,-\,measurable operators.

It was noticed in \cite{cl} that any positive $L^1-L^\ii$-\,contraction, after a unique extension, is a positive Dunford-Schwartz operator. Following this observation, various types of individual ergodic theorems for a positive Dunford-Schwartz operator in noncommutative spaces $L^p(\mc M,\tau)$ and in noncommutative Lorentz spaces $L^{p,q}(\mc M,\tau)$ were proved by first establishing maximal ergodic inequalities in $L^p(\mc M,\tau)$, $1<p<\ii$. In \cite{cl2}, utilizing the approach  that was developed in \cite{gl, cls, li}, an individual ergodic theorem for a positive Dunford-Schwartz operator in a noncommutative Orlicz space $L^\Phi(\mc M,\tau)$ with order continuous norm was proved.

In the present paper, we aim to extend noncommutative Dunford-Schwartz-type ergodic theorems in two directions:

\noindent
(A) In an attempt to answer the question how far beyond $L^1(\mc M,\tau)$ inside $L^1(\mc M,\tau)+\mc M$ can one
go for a Dunford-Schwartz-type theorem to remain valid, we have found the following. Like in the commutative case,
a noncommutative Dunford-Schwartz-type ergodic theorem, once established for a positive Dunford-Schwartz operator
in the space $L^1(\mc M,\tau)$, holds for significantly wider space of measurable operators. This symmetric space $\mc R_\tau$ consists of such operators $x\in L^1(\mc M,\tau)+\mc M$ for which $\mu_t(x)\to 0$ as $t\to \ii$, where $\mu_t(x)$ is the non-increasing rearrangement of $x$. Note that, in view of \cite{cl1}, there is a natural question whether the space $\mc R_\tau$ is maximal, that is, the largest (possibly symmetric) space for which the individual ergodic theorem holds for all hermitian Dunford-Schwartz operators.

\noindent
(B) Using the approach presented in \cite{li1}, we obtain almost uniform convergence in a variety of noncommutative individual ergodic theorems, some of which new and some previously known to hold only for the generally weaker bilaterally almost uniform convergence.

In Section 3, we show that almost uniform convergence is generally stronger than bilaterally almost convergence and present a general form of the Banach principle for the Banach space $L^p(\mc M,\tau)$, $1\leq p<\ii$, and almost uniform convergence.

A study of noncommutative Dunford-Schwartz-type individual ergodic theorems for actions of the group of integers is carried out in Section 4.

Section 5 of the paper is devoted to noncommutative Dunford-Schwartz-type ergodic theorems for actions of discrete and continuous multiparameter semigroups.

To underline the wealth of the space $\mc R_\tau$, we present, in Section 6, a variety of noncommutative symmetric spaces that are embedded in $\mc R_\tau$.

\section{Preliminaries}
Let $\mc M$ be a semifinite von Neumann algebra equipped with a faithful normal semifinite trace $\tau$. Let $\mc P(\mc M)$ stand for the lattice of projections in $\mc M$. If $\mb 1$ is the identity of $\mc M$ and $e\in\mc P(\mc M)$, we write $e^\pp=\mb 1-e$. Denote by $L^0=L^0(\mc M,\tau)$ the $*$\,-\,algebra of $\tau$\,-\,measurable operators affiliated with $\mc M$ \cite{ne}.

For every subset $E\subset L^0$, the set of all positive operators in $E$ will be denoted by $E_+$. The partial order in $L^0$ is given by its cone $L^0_+$ and is denoted by $\leq$.

Let $\| \cdot \|_\ii$ be the uniform norm in $\mc M$. Equipped with the  measure topology $t_\tau$ given by the system of (closed) neighborhoods of zero
\[
V(\ep,\dt)=\{ x\in L^0: \ \| xe\|_{\ii}\leq \dt \text{ \ for some \ } e\in \mc P(\mc M) \text{ \ with \ } \tau(e^{\pp})\leq \ep \},
\]
$\ep>0$, $\dt>0$, $L^0$ is a complete metrizable topological $*$\,-\,algebra \cite{ne}.

Let $x\in L^0$, and let $\{e_\lb\}_{\lb\ge0}$ be the spectral family of projections for the absolute value $| x|= (x^*x)^{1/2}$ of $x$. If $t>0$, then the {\it $t$-th generalized singular number of $x$} ({\it the non-increasing rearrangement of $x$}) is defined as
\[
\mu_t(x)=\inf\big\{\lb>0: \ \tau(e_\lb^\pp)\leq t\big\}
\]
(see, for example, \cite{fk}).

A non-zero linear  subspace  $E \su L^0$ with a Banach norm  $\|\cdot\|_E$ is called  {\it symmetric} ({\it fully symmetric}) on $(\mc M, \tau)$ if conditions
\[
x\in E, \ y\in L^0, \ \mu_t(y)\leq \mu_t(x) \text{ \ for all \ }t>0\]
\[
(respectively, \ x\in E, \ y\in L^0, \ \int \limits_0^s\mu_t(y)dt\leq  \int \limits_0^s\mu_t(x)dt \text{ \ for all\ \,}s>0
\]
\[
(\text{writing \ }y \prec\prec x))
\]
imply that $y\in E$ and $\| y\|_E\leq \| x\|_E$.

Let $L^0(0,\ii)$ be the linear space of (equivalence classes of) almost everywhere (a.e.) finite complex-valued Lebesgue measurable functions on the interval $(0,\ii)$, and let $\tau_\mu$ be the trace on $L^\ii(0,\ii)$ given by the integration with respect to Lebesgue measure $\mu$.

Give a symmetric function space $E\su L^0((0,\ii),\tau_\mu)$), define
\[
E(\mc M)=E(\mc M, \tau)=\big\{ x\in L^0: \ \mu_t(x)\in E\big\}
\]
and set
\[
\| x\|_{E(\mc M)}=\| \mu_t(x)\|_E,  \ x\in E(\mc M).
\]
It is shown in \cite{ks} that $(E(\mc M), \| \cdot \|_{E(\mc M)})$ is a symmetric space on $(\mc M, \tau)$.
If $E=L^p((0,\ii),\tau_\mu)$, $1\leq p<\ii$, then the space $(E(\mc M), \| \cdot \|_{E(\mc M)})$ coincides with the noncommutative $L^p$-\,space $L^p=L^p(\mc M)=(L^p(\mc M, \tau),\|\cdot \|_p)$. In addition, $L^\ii(\mc M)=\mc M$ and
\[
 (L^1\cap L^{\infty})(\mc M) = L^1(\mc M)\cap \mc M \text{ \ \ with \ \ } \|x\|_{L^1\cap\mc M}=\max \left \{\|x\|_1, \|x\|_\ii\right\},
\]
\[
(L^1 + L^\ii)(\mc M) = L^1(\mc M) + \mc M \text{ \ \ with \ \ }
\]
\[
\|x\|_{L^1+\mc M}=\inf\big\{ \|y\|_1+ \|z\|_\ii: \ x = y + z, \ y\in L^1, \ z \in\mc M\big\}=\int_0^1 \mu_t(x)\,dt
\]
(see \cite[Proposition 2.5]{ddp}). (For a comprehensive review of noncommutative $L^p$\,-\,spaces, see \cite{px, ye0}.) Note that $(L^1+\mc M , \|\cdot\|_{L^1+\mc M})$ is a fully symmetric space with Fatou property \cite[\S\,4]{dp}.

Since for a symmetric function space $E=E(0,\ii)$,
\[
L^1(0,\ii)\cap L^\ii(0,\ii)\su E \su L^1(0,\ii)+L^\ii(0,\ii)
\]
with continuous embeddings \cite[Ch.\,II, \S\,4, Theorem 4.1]{kps}, we also have
\[
L^1(\mc M)\cap\mc M\su E(\mc M)\su L^1(\mc M)+\mc M,
\]
with continuous embeddings.

Denote
\[
\mc R_\tau=\big\{x \in L^1+\mc M: \ \mu_t(x) \to 0 \text{ \ as \ } t\to \ii \big\}.
\]
Observe that $\mc R_\tau$ is a linear subspace of $L^1+\mc M$, $L^1\cap\mc M\su\mc R_\tau$, and $x\in\mc R_\tau$  if and only  if $x\in L^1+\mc M$ and $\tau\{|x|>\lb\}<\ii$ for all $\lb>0$.

The next proposition implies that $(\mc R_\tau, \|\cdot\|_{L^1+\mc M})$ is a Banach space.

\begin{pro}\cite[Proposition 2.7]{ddp}\label{p11}
$\mc R_\tau$ is the closure of $L^1 \cap \mc M$  in $L^1 + \mc M$.
\end{pro}

One can verify that if $x\in\mc R_\tau$, $y\in L^1+\mc M$ and $y\prec\prec x$, then $y\in \mc R_\tau$ and $\| y\|_{L^1+\mc M} \leq \| x\|_{L^1+\mc M}$. Therefore
$(\mc R_\tau, \|\cdot\|_{L^1+\mc M})$ is a noncommutative  fully symmetric space.

\vskip 5pt
Note that if $\tau(\mathbf 1) < \infty$, then $\mc M \subset L^1$ and $\mc R_\tau = L^1$.

\begin{pro}\label{p12}
If $\tau(\mb 1)=\ii$, then a symmetric space $E=E(\mc M,\tau)$ is contained in $\mc R_\tau$ if and only if $\mb 1\notin E$.
\end{pro}
\begin{proof}
As $\tau(\mb 1)=\ii$, we have $\mu_t(\mb 1) = 1$ for all $t > 0$, hence $\mb 1 \notin \mc R_\tau$. Therefore
$E$ is not contained in $\mc R_\tau$ whenever $\mb 1 \in E$.

Let $\mb 1\notin  E$. If $x\in E$ and $\lim\limits_{t\to\ii}\mu_t(x)=\al>0$, then $\mu_t(\mb 1)\equiv1\leq\al^{-1}\mu_t(x)$, implying $\mb 1\in E$, a contradiction. Thus $\mb 1\notin E$ entails $ E\su\mc R_\tau$.
\end{proof}

\section{A Banach principle for $L^p(\mc M,\tau)$ and almost uniform convergence}

Let $\mc N$ be a commutative von Neumann algebra equipped with a faithful normal finite trace $\mu$, and let $X$ be a Banach space. The classical Banach principle asserts that  if $M_n:X\to (L^0(\mc N, \mu),t_\mu)$ is a sequence of continuous linear maps such that
\[
\sup_n|M_n(x)|\in L^0(\mc N, \mu) \text{ \  for all\ \,} x\in X,
\]
then the set
\[
\big\{ x\in X: \ \{ M_n(x)\} \text{\ \,converges a.e.}\big\}
\]
is closed in $X$.

Theorem \ref{tau1} below is a noncommutative Banach principle for a semifinite von Neumann algebra and
almost uniform convergence in Egorov's sense.
\begin{df}
A net  $\{ x_\al\}_{\al\in A}\su L^0=L^0(\mc M,\tau)$ is said to converge to $x\in L^0$ {\it almost uniformly} ({\it a.u.}) ({\it bilaterally almost uniformly} ({\it b.a.u.})) if for every $\ep>0$ there exists $e\in\mc P(\mc M)$ such that $\tau(e^\pp)\leq\ep$ and $\lim\limits_{\al\in A}\| (x-x_\al)e\|_\ii= 0$ \ (respectively, $\lim\limits_{\al\in A}\| e(x-x_\al)e\|_\ii= 0$).
\end{df}

It is clear that a.u. convergence implies b.a.u. convergence. If $\mc M$ is commutative, a.u. and b.a.u. convergences coincide, which is not the case if $\mc M$ is not commutative:

\begin{ex}\label{e31}
Let $\mc M$ be a type $\operatorname {II}_1$ factor with $\tau(\mb 1)=1$. Let $\mc B$ be a maximal commutative $*$\,-\,subalgebra of $\mc M$. Choose a sequence $\{g_n\}_{n=1}^\ii\su\mc P(\mc B)$ such that $\tau(g_n)\to 0$ but $g_n\not\to 0$ a.u. and consider a sequence $\{e_n\}_{n=1}^\ii\su\mc P(\mc B)$ such that $e_n\downarrow 0$ and $\tau(g_n)=\tau(e_n)$ for each $n$. Since $\mc M$ is a factor and $\tau(g_n)=\tau(e_n)$, it follows that projections $g_n$ and $e_n$ are equivalent. Consequently, for every $n$ there exists a partial isometry $v_n\in\mc M$ such that $v_n v_n^* =g_n$ and $v_n^*v_n=e_n$.

We will show that $v_n^*\to 0$ b.a.u. but $v_n^*\not\to 0$ a.u. Indeed, since $e_n^\pp\uparrow\mb 1$ and $\|v_ne_n^\pp\|_\ii=0$ for every $n$, we have $v_n\to 0$ a.u., hence $v_n\to 0$ b.a.u. and so $v_n^*\to 0$ b.a.u. To show that $v_n^*\not\to 0$ a.u., let us assume that it does. Then for any $\ep >0$ there exists $q\in\mc P(\mc M)$ for which $\tau(q^\pp)\leq\ep$ and $\|v_n^*q\|_\ii\to 0$.

Let $\Phi$ be a conditional expectation of $\mc M$ onto the von Neumann algebra $\mc B$, and let $h=\big\{\Phi(q)\ge1/2\big\}$. Then we have
\[
\begin{split}
1-\ep&\leq\tau(q)=\tau(\Phi(q))=\tau(\Phi(q)h)+\tau(\Phi(q)h^\pp)<\tau(\Phi(qh))+2^{-1}\tau(h^\pp)\\
&=\tau(qh)+2^{-1}\tau(h^\pp)=\tau(hqh)+2^{-1}\tau(h^\pp)\leq\tau (h)+2^{-1}(1-\tau (h)),
\end{split}
\]
implying that $\tau(h^\pp)\leq2\ep$. Besides, $h\leq2\Phi(q)$ entails that
\[
\begin{split}
\frac12\|g_nhg_n\|_\ii&\leq\|g_n\Phi(q)g_n\|_\ii=\|\Phi(g_nq)\|_\ii\\
&=\|g_nq\|_\ii=\|v_nv_n^*q\|_\ii\leq\|v_n^*q\|_\ii\to 0.
\end{split}
\]
Therefore, we have $\|g_nh\|_\ii\to0$, hence $g_n\to0$ a.u., a contradiction.
\end{ex}

In what follows the notion of bilaterally uniform equicontinuity in measure at zero of a sequence of maps from a normed
space into the space $L^0(\mc M,\tau)$ plays an important role:

\begin{df}
Let   $(X, \| \cdot \|)$ be a normed space, and let $X_0\su X$ be such that the neutral element of $X$ is an accumulation point of $X_0$. A sequence of maps $M_n: X\to L^0$ is  called {\it bilaterally uniformly equicontinuous in measure} ({\it b.u.e.m.}) {\it at zero} in $X_0$ if for every $\ep>0$ and $\dt>0$ there exists $\gm>0$ such that, given $x\in X_0$ with $\| x\|<\gm$, there is a projection $e\in \mc P(\mc M)$  satisfying conditions
\[
\tau(e^\pp)\leq \ep \text{ \ \ and \ \ }  \sup_n\| eM_n(x)e\|_\ii\leq \dt.
\]
\end{df}

\begin{rem}
It is easy to see \cite[Proposition 1.1]{li} that, in the commutative case, bilaterally uniform equicontinuity in measure at zero in $X$ of a sequence $M_n: X\to L^0$ is equivalent to the continuity in measure at zero of the maximal operator
\[
M^*(f)=\sup \limits_n |M_n(f)|, \ f\in X.
\]
\end{rem}

The following proposition provides a critical tool for proving a.u. convergence of ergodic averages. Presented proof is a modification of the proof given in \cite{li1}.

\begin{pro}\label{pau}
Let $(X, \|\cdot\|)$ be a Banach space, $M_n:X\to L^0$ a sequence of linear maps that is b.u.e.m. at zero in $X$.
Then the set
\[
\mc L=\big\{x\in X:\ \{ M_n(x)\} \text{\ converges a.u.}\big\}
\]
is closed in  $X$.
\end{pro}
\begin{proof}
Let a sequence $\{ z_m\}\su\mc C$ and $x\in X$ be such that $\|z_m-x\|\to 0$.
Denote $y_m=z_m-x$ and fix $\ep>0$ and $\dt>0$. 

For a given $m$, as $x+y_m\in\mc C$, the sequence $\{M_n(x+y_m)\}$ converges a.u., implying that there exists $h_m\in\mc P(\mc M)$ and $N_m\in\mathbb N$ such that
\[
\tau(h_m^\perp)\leq\frac\ep{2^{m+1}}\text{ \ and \ } \|(M_n(x+y_m)-M_{n'}(x+y_m))h_m\|_\ii\leq\frac\dt{3\cdot2^m}\ \ \forall \ n,n'\ge N_m.
\]
(Without loss of generality, assume that $N_1\leq N_2\leq\dots\ $.)  

\noindent
If $h=\bigwedge\limits_mh_m$, then it follows that
\[
\tau(h^\perp)\leq\frac\ep2\text{ \ \ and \ \ }\|(M_n(x+y_m)-M_{n'}(x+y_m))h\|_\ii\leq\frac\dt{3\cdot2^m}\ \ \forall \ n,n'\ge N_m.
\]

Next, since $\|y_m\|\to0$ and $\{M_n\}$ is b.u.e.m. at zero on $X$, for every $k\in\mathbb N$ there exists $y_{m_k}=:x_k$ and $g_k\in\mc P(\mc M)$ such that
\[
\tau(g_k^\perp)\leq\frac\ep{N_{m_{k+1}}\,2^{k+2}}\text{ \ \ and\ \ } \sup_n\|g_kM_n(x_k)g_k\|_\ii\leq\frac\dt{3\cdot2^{m_k}}
\]
(without loss of generality, $m_1\leq m_2\leq\dots$).

Let $\mathbf l(y)$ ($\mathbf r(y)$) be the left (respectively, right) support of an operator $y\in L^0$. 
Set $q_{k,n}=\mathbf 1-\mathbf r(g_k^\perp M_n(x_k))$. Since for any $y\in L^0$ the projections 
$\mathbf l(y)\in \mc P(\mc M)$ and $\mathbf r(y)\in \mc P(\mc M)$ are equivalent, it follows that
\[
\tau(q_{k,n}^\perp)=\tau(\mathbf r(g_k^\perp M_n(x_k)))=\tau(\mathbf l (g_k^\perp M_n(x_k)))\leq \tau(g_k^\perp)\leq\frac \ep{N_{m_{k+1}}\,2^{k+2}}.
\]
Also,
\[
M_n(x_k) q_{k,n} =g_kM_n(x_k) q_{k,n} +g_k^\perp  M_n(x_k) q_{k,n}=g_kM_n(x_k) q_{k,n}.
\]
Therefore, letting $g_{k,n}=g_k\land q_{k,n}$, we obtain $\tau(g_{k,n}^\perp)\leq\displaystyle\frac\ep{N_{m_{k+1}}\,2^{k+1}}$ and
\[
M_n(x_k)g_{k,n}=M_n(x_k)q_{k,n}g_{k,n}=g_kM_n(x_k)q_{k,n}g_{k,n} = g_kM_n(x_k) g_kg_{k,n},
\]
implying
\[
\sup_n\|M_n(x_k)g_{k,n}\|_\ii\leq\|g_kM_n(x_k)g_k\|_\ii\leq\frac\dt{3\cdot2^{m_k}}.
\]

If we put $f_k=\bigwedge\limits_{n\leq N_{m_{k+1}}}g_{k,n}$, then it follows that
\[
\tau(f_k^\perp)\leq\frac \ep{2^{k+1}}\text{ \ \ and \ \ } \sup_{n\leq N_{m_{k+1}}}\| M_n(x_k) f_k\|_\ii\leq\frac\dt{3\cdot2^{m_k}}\text{ \ \ for all \ } k.
\]
Further, letting $f=\bigwedge\limits_kf_k$, we arrive at
\[
\tau(f^\perp)\leq\frac\ep2\text{ \ \ and \ \ }\sup_{n\leq N_{m_{k+1}}}\|M_n(x_k)f\|_\ii\leq\frac\dt{3\cdot2^{m_k}}\text{ \ \ for all \ } k.
\]

Now, let $e=h\land f$. Then $\tau(e^\perp)\leq\ep$ and $N_{m_k}\leq n, n'\leq N_{m_{k+1}}$, in view of $x_k=y_{m_k}$, implies that
\[
\begin{split}
\|(M_n(x)-M_{n'}(x))e\|_\ii&\leq\|(M_n(x+x_k)-M_{n'}(x+x_k))h\|_\ii\\
&+\|M_n(x_k)f\|_\ii+\|M_{n'}(x_k)f\|_\ii\leq\frac\dt{2^{m_k}}.
\end{split}
\]
Therefore, since $m_1\leq m_2\leq\dots$ implies, as  $N_1\leq N_2\leq\dots$, that $N_{m_1}\leq N_{m_2}\leq\dots$, we have
\[
\|(M_n(x)-M_{n'}(x))e\|_\ii\leq\sum_{k=1}^\ii\|(M_{N_{m_{k+1}}}(x)-M_{N_{m_k}}(x))e\|_\ii\leq\sum_{k=1}^\ii\frac\dt{2^{m_k}}\leq\dt
\]
for all $n, n'\ge N_{m_1}$ (we can assume without loss of generality that $N_{m_1}=n<n'=N_{m_{k_0}}$ for some $k_0$). Therefore, the sequence $\{M_n(x)\}$ is a.u. Cauchy in $L^0$. Since $L^0$ is complete with respect to a.u. convergence \cite[Remark 2.4]{cls}, we have $x\in\mc C$, which completes the argument.
\end{proof}

When $X=L^p$, $1\leq p<\ii$, we can say more:

\begin{teo}\label{tau1}
Let $1\leq p<\ii$ and $M_n: L^p\to (L^0,t_\tau)$ be a sequence of positive continuous linear maps such that, given $x\in L^p$ and
$\ep>0$, there exists $e\in\mc P(\mc M)$ satisfying conditions
\[
\tau(e^\pp)\leq \ep \text{ \ \ and \ \ } \sup_n\|eM_n(x)e\|_\ii<\ii.
\]
Then the set
\[
\big\{ x\in L^p: \ \{ M_n(x)\} \text{\ converges a.u.}\big\}
\]
is closed in $L^p$.
\end{teo}

\begin{proof}
By \cite[Theorem 3.2]{cl0} with $E=X_+=L^p_+$, the sequence $\{M_n\}$ is b.u.e.m. at zero in $L^p_+$, which, by
\cite[Lemma 4.1]{li}, implies that $\{M_n\}$ is b.u.e.m. at zero in $L^p$, and the result follows by Proposition \ref{pau}.
\end{proof}

\section{Single parameter individual ergodic theorems in $\mc R_\tau$}

A linear operator $T: L^1+\mc M\to L^1+\mc M$ is called a {\it Dunford-Schwartz operator} (writing $T\in DS)$ if
$T(L^1)\su L^1$, $T(\mc M)\su\mc M$ and
\[
\|T(x)\|_1 \leq\| x\|_1\text{ \ for all\ \,}x\in L^1\text{ \ and \ }\| T(x)\|_\ii \leq\| x\|_\ii\text{ \ for all\ \,}x\in\mc M.
\]
If a Dunford-Schwartz operator $T$ is  positive, that is, $T(x)\ge 0$ whenever $x\ge 0$, we shall write $T\in DS^+$.

If $T\in DS$, then $T(x )\prec\prec x$ for all $x\in L^1+\mc M$ \cite[Theorem 4.7]{ddp}, implying that
$T(E)\su E$ and $\| T\|_{E\to E}\leq1$ for every fully symmetric space $E=E(\mc M,\tau)$.

Given $T\in DS$ and $x\in L^1+\mc M$, denote
\begin{equation}\label{e13}
A_n(x)=A_n(T)( x)=\frac 1n\sum_{k=0}^{n-1} T^k(x), \ \ n=1,2,\dots,
\end{equation}
the corresponding Ces\'aro averages of the operator $x$.

\vskip 5pt
The following noncommutative individual ergodic theorem was established in \cite[Theorem 2]{ye}.

\begin{teo}\label{t12}
Let $T\in DS^+$ and $x\in L^1$. Then the averages (\ref{e13}) converge b.a.u. to some $\wh x\in L^1$.
\end{teo}

An extension of Theorem \ref{t12} to noncommutative spaces $L^p$, $1<p<\ii$, was obtained in \cite{jx} (see also \cite{cl} and \cite{li}):

\begin{teo}\label{t121}
If $T\in DS^+$ and $x\in L^p$, $1<p<\ii$, then the averages (\ref{e13}) converge b.a.u. to some $\wh x\in L^p$.
If $p\ge 2$, then these averages converge also a.u.
\end{teo}

It was stated in \cite[Theorem 1.3]{li1} that, in fact, we have a.u. convergence in Theorems \ref{t12} and \ref{t121}:
\begin{teo}\label{t43}
Let $T\in DS^+$ and $x\in L^p$, $1\leq p<\ii$. Then the averages (\ref{e13}) converge a.u. to some $\wh x\in L^p$.
\end{teo}
\begin{proof}
It is well-known (see, for example, \cite[Proof of Theorem 1.5]{cl}) that the sequence $\{A_n(x)\}$ converges a.u. 
whenever $x\in L^2$. Therefore, since the set $L^p\cap L^2$ is dense in $L^p$ and, by \cite[Proposition 4.2]{li}, the sequence $M_n=A_n: X=L^p\to L^0$ is b.u.e.m. at zero in $L^p$, Proposition \ref{pau} guarantees that the averages $A_n(x)$ converge a.u. for each $x\in L^p$ (to some $\widehat x\in L^0$), hence we also have $A_n(x)\to \widehat x$ in measure. 
As each $A_n$ is a contraction in $L^p$ the unit ball of which is closed in measure topology, we conclude that 
$\widehat x\in L^p$.
\end{proof}

Here is an extension of Theorem \ref{t43} to $\mc R_\tau$:
\begin{teo}\label{t13}
Let $T\in DS^+$ and $x\in \mc R_\tau$. Then the averages (\ref{e13}) converge a.u. to some $\wh x\in  \mc R_\tau$.
\end{teo}

\begin{proof}
Without loss of generality assume that $x\geq 0$, and let $\{ e_\lb\}_{\lb\ge 0}$ be the spectral family of $x$. Given $m\in \mbb N$, denote $x_m= \int_{1/m}^\ii\lb\,de_\lb$ and $y_m = \int_0^{1/m}\lb\,de_\lb$. Then $0 \leq y_m \leq \frac1m\,\mb 1$, $x_m \in L^1$, and $x = x_m+y_m$ for all $m$.

Fix $\ep > 0$. By Theorem \ref{t43}, $A_n(x_m) \to \wh x_m \in L^1$ a.u. for each $m$.
Therefore there exists $e_m\in \mc P(\mc M)$ such that
\[
\tau(e_m^{\pp})\leq\frac\ep{2^m}\text{ \ \ and \ \ }
\| (A_n(x_m)-\wh x_m)e_m\|_\ii\to 0
\]
as $n \to \ii$.
Then it follows that
\[
\|(A_k(x_m)-A_l(x_m))e_m\|_\ii< \frac1m \text{ \ \ for all \ }  k, l \geq N(m).
\]
Since $\|y_m\|_\ii\leq 1/m$, we have
\[
\begin{split}
\| (A_k(x)-A_l(x))e_m\|_\ii
&\leq \| (A_k(x_m)-A_l(x_m))e_m\|_\ii+\| (A_k(y_m)-A_l(y_m))e_m\|_\ii\\
&<\frac1m+\|A_k(y_m)e_m\|_\ii+\| A_l(y_m)e_m\|_\ii\leq\frac3m
\end{split}
\]
for each $m$ and all $k, l \geq N(m)$. So, if $e=\bigwedge\limits_m e_m$, then
\[
\tau(e^\pp)\leq\ep\text{ \  \ and \ \ } \|(A_k(x)-A_l(x))e\|_\ii<\frac3m\text{ \ \ for all \ }k, l\ge N(m).
\]
This means that $\{A_n(x)\}$ is a Cauchy sequence with respect to a.u. convergence.
Since $L^0$ is complete with respect to a.u. convergence \cite[Remark 2.4]{cls}, we conclude that the sequence
$\{A_n(x)\}$ converges a.u. to some $\wh x \in L^0$.

As $A_n(x)\to \wh x$ a.u., it is clear that $A_n(x)\to \wh x$ in measure. Since  the unit ball in $L^1+\mc M$ is closed in measure topology \cite[Theorem 4.1]{ddst}, it follows that $\wh x\in L^1+\mc M$. In addition, measure convergence $A_n(x)\to\wh x$ implies that $\mu_t(A_n(x))\to \mu_t(\wh x)$ a.e. on $((0,\ii),\mu)$ (this can be shown as in the commutative case; see, for example, \cite[Ch.\,II, \S\,2, Property 11$^\circ$]{kps}).

Since  $T\in DS^+$, we have $A_n(T) \in DS^+$,  hence $\mu_t(A_n(x))\prec\prec \mu_t( x)$, that is,
\[
\int \limits_0^s\mu_t(A_n(x))dt\leq  \int \limits_0^s\mu_t(x)dt \text{ \ for all \ }s>0, \,n\in\mbb N.
\]
Since $\mu_t(A_n(x))\to\mu_t(\wh x)$ a.e. on $((0,s),\mu)$, Fatou lemma implies that
\[
\int\limits_0^s\mu_t(\wh x)dt\leq  \int \limits_0^s\mu_t(x)dt \text{ \ for all\ }s>0,
\]
that is, $\mu_t(\wh x)\prec\prec \mu_t( x)$. Since $\mc R_\tau$ is a fully symmetric space and $x\in \mc R_\tau$, it follows that $\wh x\in\mc R_\tau$.
\end{proof}

Next we give an application of Theorem \ref{t13} to fully symmetric spaces:

\begin{teo}\label{t46}
Let $\tau(\mb 1)=\ii$, and let $E=E(\mc M,\tau)$ be a fully symmetric space such that $\mb 1\notin E$. If $T \in DS^+$, then for every $x\in E$ the averages (\ref{e13}) converge a.u. to some $\wh x\in E$.
\end{teo}

\begin{proof}
By Proposition \ref{p12}, $E \su  \mc R_\tau$. Then it follows from Theorem  \ref{t13} that the averages
$A_n(x)$ converge a.u. to some $\wh x \in \mc R_\tau$.  Since $\mu_t(\wh x)\prec\prec \mu_t( x)$
(see the proof of Theorem \ref{t13}) and $E$ is a fully symmetric space, we obtain $\wh x\in E$.
\end{proof}

Now, let $\Bbb C_1=\{z\in \Bbb C: |z|=1\}$ be the unit circle in $\mbb C$. A function $P : \mbb Z\to\mbb C$
is said to be a {\it trigonometric polynomial} if $P(k)=\sum_{j=1}^{s} z_j\lb_j^k$, $k\in \mbb Z$, for some $s\in \mbb N$, $\{ z_j \}_1^s \subset \mbb C$, and $\{\lb_j \}_1^s \subset \mbb C_1$.
A sequence $\{ \beta_k \}_{k=0}^{\ii} \subset \Bbb C$ is called a {\it bounded Besicovitch sequence} if

(i) $| \beta_k | \leq C < \ii$ for all $k$;

(ii) for every $\ep >0$ there exists a trigonometric polynomial $P$ such that
\[
\limsup_n \frac 1n \sum_{k=0}^{n-1} | \beta_k - P(k) | < \ep.
\]

The following theorem was proved in \cite{cls}.

\begin{teo}\label{t14}
Assume that $\mc M$ has a separable predual. Let $T\in DS^+$, and let $\{ \bt_k\}$ be a bounded Besicovitch sequence.
Then for every $x\in L^1(\mc M)$ the averages
\begin{equation}\label{e14}
B_n(x)=\frac 1n\sum_{k=0}^{n-1}\bt_kT^k(x)
\end{equation}
converge b.a.u. to some $\wh x\in L^1(\mc M)$.
\end{teo}

Below is an extension of Theorem \ref{t14} to the fully symmetric space $\mc R_\mu$.

\begin{teo}\label{t15}
Let $\mc M$, $T$, and $\{ \bt_k\}$ be as in Theorem \ref{t14}.
Then for every $x\in \mc R_\tau$ the averages (\ref{e14})
converge a.u. to some $\wh x\in \mc R_\tau$.
\end{teo}
\begin{proof}
First, using Theorem \ref{t43} instead of \cite[Theorem 1.4]{cls} in the proof of \cite[Lemma 4.2]{cls}, we conclude that
the averages
\[
\widetilde B_n(x)=\frac 1n \sum_{k=0}^{n-1}P(k)T^k(x)
\]
converge a.u. for every trigonometric polynomial $P$ and $x\in L^1$.

Next, it follows exactly as in \cite[Theorem 4.4]{cls} that the sequence $\{ B_n(x)\}$ converges a.u.
for every $x\in L^1\cap \mc M$.

By \cite[Theorem 2.1]{cl2}, the sequence $B_n:L^1\to L^0$ is b.u.e.m. at zero in $L^1$, which, in view of Proposition \ref{pau}, implies that the set
\[
\big\{x\in L^1:\, \{ B_n(x)\} \text{\ \,converges a.u.}\big\}
\]
is closed in $L^1$. Therefore, as $L^1\cap \mc M$ is dense in $L^1$, we conclude that the sequence $\{ B_n(x)\}$
converges a.u. for all $x\in L^1$.

Now, as $\sup_k|\bt_k|\leq C$, we have $C^{-1}B_n\in DS$, hence $\mu_t\left(C^{-1}B_n(x)\right)\prec\prec\mu_t( x)$ for every $x\in\mc R_\tau$. Then, as in the proof of Theorem \ref{t13}, we obtain  that for every $x\in \mc R_\tau$ the sequence $\{ B_n(x)\}$
converges a.u. to some  $\wh x\in \mc R_\tau$.

\end{proof}

\section{Multiparameter individual ergodic theorems in $\mc R_\tau$}

Fix $d\in\mbb N$. If $\mbb N_0=\{ 0,1,2, \dots \}$, let $\mbb N_0^d$ be the additive semigroup of $d$\,-\,dimensional vectors
over $\mbb N_0$.

We will write $\mbb N_0^d\ni \mb n=(n_1,\dots, n_d)\to \ii$ whenever $n_i\to \ii$ for each $1\leq i\leq d$.
Accordingly, if $\mb n(k)=( n_1(k), \dots , n_d(k))$, $k=1,2,\dots$, is a sequence in $\mbb N_0^d$,
then we write $\mb n(k)\to\ii$ if $n_i(k)\to \ii$ for each $1\leq i\leq d$.

Given $\mb n=(n_1,\dots,n_d)\in \mbb N_0^d$, $\{ T_i\}_{i=1}^d\su DS$, and $x\in L^1+\mc M$,  denote
\begin{equation}\label{emulti}
A_{\mb n}(x)=\frac 1{n_1\cdot \dotsc \cdot n_d}\sum_{i_1=0}^{n_1-1}\dotsc \sum_{i_d=0}^{n_d-1}
T_1^{i_1}\cdot \dotsc \cdot T_d^{i_d}(x),
\end{equation}
the corresponding multiparameter ergodic averages.

A multiparameter sequence  $\{ x_{\mb n}\}\su L^0$ is said to converge to $x\in L^0$ {\it almost uniformly} ({\it a.u.})
({\it bilaterally almost uniformly} ({\it b.a.u.})) if for every $\ep>0$ there exists such a projection
$e\in \mc P(\mc M)$ such that $\tau(e^{\pp})\leq \ep$ and
$\| (x-x_{\mb n})e\|_{\ii}\to 0$ (respectively, $\| e(x-x_{\mb n})e\|_{\ii}\to 0$) as $\mb n\to \ii$.

It is known \cite[Theorem 6.6]{jx}  (see also \cite{gg,pe,sk}) that if $1<p<\ii$, then the averages (\ref{emulti})
converge b.a.u. (and a.u. if $p> 2$).  The following theorem asserts that we have a.u. convergence for all $1<p<\ii$.

\begin{teo}\label{tau}
Given $1<p\leq2$, the averages (\ref{emulti}) converge a.u. for every $x\in L^p$.
\end{teo}
\begin{proof}
It is enough to show convergence for any sequence $\mb n(k)\to \ii$ (see the last paragraph in the proof of Theorem \ref{tsec}).

By \cite[Theorem 6.6]{jx}, the sequence $\{ A_{\mb n(k)}(x)\}$ converges b.a.u. for every $x\in L^p$. Therefore,
given $x\in L^p$ and $\ep>0$, there is $e\in \mc P(\mc M)$ such that
\[
\tau(e^\pp)\leq \ep \text{ \ \ and \ \ } \sup_k\|eA_{\mb n(k)}(x)e\|<\ii.
\]
Thus, by Theorem \ref{tau1}, the set
\[
\big\{ x\in L^p: \,\{ A_{\mb n(k)}(x)\} \text{\ \,converges a.u.}\big\}
\]
is closed in $L^p$. This completes the proof since, given any $q> 2$, the set $L^p\cap L^q$ is dense in $L^p$,
and the sequence $\{ A_{\mb n(k)}(x)\}$ converges a.u. whenever $x\in L^p\cap L^q$ by \cite[Theorem 6.6]{jx}.
\end{proof}

Now we turn to multiparameter individual ergodic theorems in $\mc R_\tau$. When $p=1$ the above {\it unrestricted} pointwise convergence of the averages (\ref{emulti}) is known to not always hold even if $T_i$'s are given by commuting measure preserving transformations.

We will begin with proving a.u. convergence of the averages (\ref{emulti}) for all $x\in L^1$ under the assumption that the sequence $\mb n(k)\to \ii$ in question remains in a sector of $\mbb N_0^d$ (see Theorem \ref{tsec} below).

\begin{df}
A sequence $\{ \mb n(k)\} \su \mbb N_0^d$ is said to {\it remain in a sector of $\mbb N_0^d$} if there exists
a constant $0<c_0<\ii$ such that $\displaystyle\frac{n_i(k)}{n_j(k)}\leq c_0$ for all $1\leq i,j\leq d$ and $k$ \ ($n_j(k) \neq 0$).
\end{df}

Our argument is based on the following result due to A. Brunel \cite{br} (see \cite[Ch.\,6, \S\,6.3, Theorem 3.4]{kr}).
Although it was originally formulated for commuting contractions in a commutative $L^1$-\,space, one can see that the proof applies when $L^1=L^1(\mc M,\tau)$.

\begin{teo}\label{tbr}
There exist a constant $\chi_d>0$ and a family $\{ a_{\mb n}>0: \mb n\in \mbb N_0^d\}$
with $\sum a_{\mb n}=1$ such that, for
any positive commuting contractions $T_1,\dots,T_d$ of $L^1$, the operator
$S=\sum a_{\mb n}T_1^{n_1}\cdot \dotsc \cdot T_d^{n_d}$ satisfies the inequality
\begin{equation}\label{ebr}
\frac 1{n^d}\sum_{i_1=0}^{n-1}\dotsc \sum_{i_d=0}^{n-1}T_1^{i_1}\cdot \dotsc \cdot T_d^{i_d}(x)\leq
\frac {\chi_d}{n_d}\sum_{j=0}^{n_d-1}S^j(x)
\end{equation}
for some $n_d$ and all $n=1,2,\dots$ and $x\in L_+^1$.
\end{teo}

We will need Yeadon's maximal ergodic inequality for the averages (\ref{e13}) \cite[Theorems 2.1]{cl}:

\begin{teo}\label{tye} Let $T\in DS^+$,  and let $\{ A_n\}$ be given by (\ref{e13}).
Then for every $x\in L^1_+$ and $\nu>0$
there exists a projection $e\in \mc P(\mc M)$ such that
$$
\tau(e^\pp)\leq \frac{\| x\|_1}\nu \text{ \ \ and \ \ } \sup_n\| eA_n(x)e\|_\ii\leq \nu.
$$
\end{teo}

The next result is a noncommutative extension of  \cite[Ch.6, \S 6.3, Theorem 3.5]{kr}.

\begin{teo}\label{tsec}
Let $\{ T_1, \dots, T_d\}\su DS^+$ be a commuting family. Then for every $x\in L^1$ the averages $A_{\mb n(k)}(x)$ given by (\ref{emulti}) converge a.u. to some $\wh x\in L^1$ for any sequence $\mb n(k)\to \ii$
that remains in a sector of \ $\mbb N_0^d$.
\end{teo}
\begin{proof}
Assume that $x\in L^1_+$ and let $m_k=\max\{n_1(k),\dots, n_d(k)\}$, $k=1,2,\dots$ Since the sequence  $\{ \mb n(k)\}$
remains in a sector of $\mbb N_0^d$, there is $0<c_0<\ii$ such that $m_k\leq c_0n_i(k)$ for all $1\leq i\leq d$ and $k$.
Then, in accordance with Theorem \ref{tbr}, we have
\[
0 \leq A_{\mb n(k)}(x)\leq\frac{c_0^d}{m_k^d}\sum_{i_1=0}^{m_k-1}\dotsc\sum_{i_d=0}^{m_k-1}T_1^{i_1}\cdot \dotsc \cdot
T_d^{i_d}(x)\leq \frac {\chi_dc_0^d}{(m_k)_d}\sum_{j=0}^{(m_k)_d-1}S^j(x).
\]
As $S\in DS^+$, it follows from Theorem \ref{tye} that for a given $\nu>0$ there exists $e\in \mc P(\mc M)$ such that
\[
\tau(e^\pp)\leq \frac{\| x\|_1}\nu \text{ \ \ and \ \ } \sup_k\left \|e\frac 1{(m_k)_d}\sum_{j=0}^{(m_k)_d-1}S^j(x)e\right \|_\ii \leq \nu,
\]
which implies that
\[
\sup_k \| eA_{\mb n(k)}(x)e\|_\ii \leq \chi_dc_0^d\,\nu.
\]
Therefore the sequence $\{ A_{\mb n(k)}\}$ is b.u.e.m. at zero on $L^1_+$, which, by \cite[Lemma 4.1]{li}, implies that
it is also b.u.e.m. at zero on $L^1$. This in turn entails, by Proposition \ref{pau}, that we only need to have a.u. convergence
of the sequence $\{ A_{\mb n(k)}(x)\}$ for every $x$ in a dense subset of $L^1$. Let $p>2$. Since the set
$L^1\cap L^p$ is dense in $L^1$, this property is granted by \cite[Theorem 6.6]{jx}, and we conclude that
the sequence $\{ A_{\mb n(k)}(x)\}$ converges a.u. for all $x\in L^1$.

Now, if $x\in L^1$ and $\mb n(k)\to \ii$ remaining in a sector of $\mbb  N_0^d$, we can follow the argument at the end of
proof of Theorem \ref{t13} to show that $A_{\mb n(k)}(x)\to \wh x$ a.u. for some $\wh x\in L^1$.

Let $\mb m(k)$ be another sequence in $\mbb N_0^d$ that tends to $\ii$ remaining in a sector of $\mbb  N_0^d$.
Then the sequence $\mb m(1), \mb n(1), \mb m(2), \mb n(2), \dots$ also tends to $\ii$ and  remains in a sector of
$\mbb  N_0^d$ and thus the sequence $A_{\mb m(1)}(x), A_{\mb n(1)}(x) , A_{\mb m(2)}(x), A_{\mb n(2)}(x), \dots$
converges a.u. to some $\widetilde x\in L^1$. This implies that $\widetilde x=\wh x$, so $A_{\mb m(k)}(x)\to \wh x$
a.u., and the proof is complete.
\end{proof}

Now we expand Theorem \ref{tsec} to the space $\mc R_\tau$:
\begin{teo}\label{tsec0}
Let $\{ T_1, \dots, T_d\}\su DS^+$ be a commuting family. Then, given $x\in \mc R_\tau$,
the averages (\ref{emulti}) converge a.u. to some
$\wh x\in \mc R_\tau$ for any sequence $\mb n(k)\to \ii$
that remains in a sector of \ $\mbb N_0^d$.
\end{teo}
\begin{proof}
By Theorem  \ref{tsec}, the averages $A_{\mb n(k)}(x)$ converge a.u. whenever $x\in L^1$.
Repeating the proof of Theorem \ref{t13} and using
Proposition \ref{p2} below, which is just a restatement of \cite[Theorem 2.3]{cls}, one can see that
the averages $A_{\mb n(k)}(x)$ converge a.u. for each $x\in \mc R_\tau$ to some $\wh x\in \mc R_\tau$.
\end{proof}
Using completeness of $L^0$ with respect to the a.u. convergence, we obtain the following.
\begin{pro}\label{p2}
Let $\mbb N_0^d\ni \mb n(k)\to \ii$, and let $\{ x_{\mb n(k)}\}\su L^0$ be such that for every $\ep>0$ there is
$e\in \mc P(\mc M)$ with $\tau(e^\pp)\leq \ep$ for which $\| (x_{\mb n(k)}-x_{\mb n(l)})e\|_\ii\to 0$
($\| e(x_{\mb n(k)}-x_{\mb n(l)})e\|_\ii\to 0$) as $\mb n(k),\ \mb n(l)\to \ii$. Then there is $\wh x\in L^0$ such that
$x_{\mb n(k)}\to \wh x$ a.u. (respectively, b.a.u.).
\end{pro}

Next we give an application of Theorem \ref{tsec0} to fully symmetric spaces. The proof of this theorem
is similar to that of Theorem \ref{t46}.

\begin{teo}\label{tsec01}
Let $E=E(\mc M,\tau)$ be a fully symmetric space such that
$\mb 1\notin E$, and let $\{ T_1, \dots, T_d\}\su DS^+$ be a commuting family.
Then, given $x\in E$, the averages (\ref{emulti}) converge a.u. to some
$\wh x\in E$ for any sequence $\mb n(k)\to \ii$
that remains in a sector of \ $\mbb N_0^d$.
\end{teo}

\begin{cor}\label{tsec02}
Let $\{ T_1, \dots, T_d\}\su DS^+$ be a commuting family, $1\leq p<\infty$. Then for every $x\in L^p$ the averages $A_{\mb n(k)}(x)$ given by (\ref{emulti}) converge  a.u. to some $\wh x\in L^p$ for any sequence $\mb n(k)\to\ii$
that remains in a sector of $\mbb N_0^d$.
\end{cor}

Let $\{T_{\mb u}: \mb u\in \mbb R_+^d\}$ be a semigroup of linear contractions of $L^1$ which is {\it $L^1$-\,continuous} in the interior of $\mbb R_+^d$, that is,
\[
\| T_{\mb u}(x)-T_{\mb v}(x)\|_1\to 0 \text{ \ \,as \ } \mb u\to\mb v
\]
for all  $x\in L^1$ and $\mb v=(v_1,\dots, v_d)\in\mbb R_+^d$ with $v_i>0$, $1\leq i\leq d$.

Fix $x\in L^1$. Then the function $\varphi(\mb u)=\tau(T_{\mb u}(x)y)$ is continuous on $\mbb R_+^d$ for every $y\in\mc M=(L^1)^\prime$. Therefore, if $\nu$ is the Lebesgue measure on $\mbb R_+^d$, then the function $f_x: \mbb R_+^d \to L^1$ defined by $f_x(\mb u) = T_{\mb u}(x)$ is weakly $\nu$\,-\,measurable. Since $f_x(\mbb R_+^d )$ is a separable subset of $L^1$, Pettis theorem \cite[Ch.\,V, \S\,4]{yo} entails that the function $f_x$ is strongly $\nu$\,-\,measurable and the function $\|f_x({\mb u})\|_1=\|T_{\mb u}(x)\|_1$ is $\nu$\,-\,measurable on $\mbb R_+^d$. Since $\|T_{\mb u}(x)\|_1\leq\|x\|_1$, it follows that the function $\|T_{\mb u}(x)\|_1$ is integrable on $(0,t]^d$ for any $t>0$. Then, by \cite[Ch.\,V, \S\,5, Theorem 1]{yo}, the function $T_{\mb u}(x)$ is Bochner $\nu$\,-\,integrable on $(0,t]^d$, $t>0$. Therefore, for any $x\in L^1$ and $t>0$ there exists the integral
\begin{equation}\label{eflow}
A_t(x)=\frac 1{t^d}\int_{(0,t]^d}T_{\mb u}(x)d\mb u.
\end{equation}

The next theorem is a noncommutative extension of the theorem of Dunford and Schwartz \cite[Theorem VIII.7.17, p.708]{ds} (see  also \cite[\S\,6.3, Theorem 3.7]{kr}). 

\begin{teo}\label{tflow}
Let $\{ T_{\mb u}:\mb u\in \mbb R_+^d\}\su DS^+$ be a semigroup $L^1$-\,continuous on the interior of $\mbb R_+^d$. If $x \in L^1$, then the averages $A_t(x)$ given by (\ref{eflow}) converge  a.u. to some $\wh x\in L^1$.
\end{teo}
\begin{proof}
One can verify that if $x\in L^1$ and $y=\int_{(0,1]^d}T_{\mb u}(x)d\mb u$, then
\[
\int_{(0,n]^d}T_{\mb u}(x)d\mb u=\sum_{i_1=0}^{n-1}\dotsc\sum_{i_d=0}^{n-1}T_{(1,0,\dots,0)}^{i_1}\cdot \dotsc
\cdot T_{(0,\dots,0,1)}^{i_d}(y), \ \ n=1,2, \dots
\]
Therefore, by Theorem \ref{tsec} applied to $y\in L^1$, there exists $\wh x\in L^1$ such that
the  sequence $\{ A_n(x)\}_{n=1}^\ii$, given by (\ref{eflow}), converges a.u. to $\wh{x}$.

If $t>1$ and $x\in L^1\cap \mc M$, then
\begin{equation}\label{eflow1}
\begin{split}
\| A_t(x)&-A_{[t]}(x)\|_\ii=\left \| \frac 1{t^d}\int_{(0,t]^d}T_{\mb u}(x)d\mb u-
\frac 1{[t]^d}\int_{(0,[t]]^d}T_{\mb u}(x)d\mb u \right \|_\ii\\
&=\left \| \frac 1{t^d}\int_{(0,t]^d\sm [0,[t]]^d}T_{\mb u}(x)d\mb u+
\left (\frac 1{t^d}-\frac 1{[t]^d}\right )\int_{(0,[t]]^d}T_{\mb u}(x)d\mb u \right \|_\ii\\
&\leq \frac 1{t^d}(t^d-[t]^d)\| x\|_{\ii}+\left ( \frac 1{[t]^d}-\frac 1{t^d}\right )[t]^d\| x\|_{\ii} \to 0
\end{split}
\end{equation}
as $t\to \ii$.

Since $\{A_{[t]}(x)\}=\{ A_n(x)\}_{n=1}^\ii\to\wh x$ a.u., given $\ep>0$, there exists a projection $e\in \mc P(\mc M)$ such that
\[
\tau(e^\pp)\leq\ep\text { \ \ and \ \ } \| (A_{[t]}(x)-\wh x)e\|_\ii \to 0.
\]
By (\ref{eflow1}), we have
\[
\|(A_t(x)-\wh x)e\|_\ii\leq\|(A_t(x) -  A_{[t]}(x))e\|_\ii + \| (A_{[t]}(x)- \wh x)e\|_\ii\to 0
\]
as $t\to \ii$,  hence $A_t(x)\to\wh x$ \ a.u.

If $z\in L^1_+$ and $y=\int_{[0,1]^d}T_{\mb u}(z)d\mb u$, then we have as before
\[
\begin{split}
0 \leq A_t(z)&\leq\frac 1{[t]^d}\int_{(0,t]^d}T_{\mb u}(z)d\mb u\leq\frac 1{[t]^d}\int_{(0,[t]+1]^d}T_{\mb u}(z)d\mb u\\
&=\left (\frac{[t]+1}{[t]}\right )^d\frac 1{([t]+1)^d}\sum_{i_1=0}^{[t]}\dotsc\sum_{i_d=0}^{[t]}
T_{(1,0,\dots,0)}^{i_1}\cdot \dotsc \cdot T_{(0,\dots,0,1)}^{i_d}(y).
\end{split}
\]
Since $y\in L^1_+$, for every $\nu>0$, we can find, as it was carried out in the proof of Theorem \ref{tsec} with the help
of Theorem \ref{tbr}, a projection $e\in \mc P(\mc M)$ such that
\[
\tau(e^\pp)\leq\frac{2\| y\|_1}\nu\leq \frac {2\| z\|_1}\nu\text { \ \ and \ \ }
\sup_{t>0}\| eA_t(z)e\|_\ii<\chi_d\,\nu.
\]
Hence, given $z\in L_+^1$ and $\nu>0$, there exists $e\in\mc P(\mc M)$ such that
\begin{equation}\label{e2020}
\tau(e^\pp)\leq\frac{2\|z\|_1}\nu\text{ \ \ and \ \ } \sup_{t>0}\|eA_t(z)e\|_\ii<\chi_d\,\nu.
\end{equation}

Let now $x\in L^1_+$, and let $\{ e_\lb\}_{\lb\ge 0}$ be its spectral family. Given $m\in\mbb N$, denote $x_m= \int_0^m \lb de_\lb$ and $z_m = x - x_m$. Then $\{x_m\}\su L^1_+ \cap \mc M$,  $\{z_m\}\su L^1_+$ and $\|z_m\|_1 \to 0$.

Fix $\ep>0$ and $\dt>0$, and let $\nu=\frac\dt{3\chi_d}$. Since $\|z_m\|_1\to0$, there exists $m_1\in\mbb N$ such that 
$\|z_{m_1}\|_1<\frac{\ep\nu}4$. By (\ref{e2020}), there exists a projection $e\in\mc P(\mc M)$ such that
\begin{equation}\label{eflow3}
\tau(e^\pp)\leq\frac{2\|z_{m_1}\|_1}\nu\leq\frac \ep 2 \text { \ \ and \ \ } \sup_{t>0} \| eA_t( z_{m_1})e\|_\ii<\frac\dt 3.
\end{equation}
Since $x_{m_1}\in L^1\cap\mc M$, we have $A_t(x_{m_1} )\to\wh x_{m_1}$ b.a.u., so the net $\{A_t(x_{m_1})\}_{t>0}$ is b.a.u. Cauchy. Therefore, there exist $g\in \mc P(\mc M)$ and $t_1>0$ such that
\begin{equation}\label{eflow4}
\tau(g^\pp)<\frac\ep2 \ \ \text{\ and \ \ }
\|g(A_t( x_{m_1})-A_{t'}( x_{m_1}))g\|_\ii<\frac\dt3.
\end{equation}
for all $t, t' \ge t_1$.

If $h=e\wedge g$, then it follows from (\ref{eflow3})  and (\ref{eflow4}) that $\tau(h^\pp)<\ep$ and
\[
\begin{split}
\| h(A_t(x)-A_{t'}(x))h\|_\ii&\leq \|h(A_t(x_{m_1})- A_{t'}( x_{m_1}))h\|_\ii\\
&+\| hA_t(z_{m_1})h\|_\ii + \| hA_{t'}( z_{m_1})h\|_\ii < \delta.
\end{split}
\]
Thus, the net $\{A_t(x)\}_{t>0}$ is a.u. Cauchy. As $L^0$ is complete with respect to b.a.u. convergence, we conclude that the averages $A_t(x)$ converge a.u. to some $\wh x\in L^0$. As in the proof of Theorem \ref{t13}, we conclude that $\wh x\in L^1$.
\end{proof}

In particular, we have the following (cf.\ \cite[Theorem 6.8]{jx} and Remarks right after it).
\begin{cor}\label{tflow1}
Let $\{ T_s\}_{s\ge 0} \su DS^+$  be a semigroup that is strongly continuous on $L^1$ at every $s>0$. Then the averages
\[
\frac 1t\int_0^tT_s(x)ds
\]
converge a.u. for every $x\in L^1$ as $t\to \ii$.
\end{cor}

Let $\{T_{\mb u}: \mb u\in\mbb R_+^d\} \su DS^+$ be a semigroup $L^1$-\,continuous in the interior of $\mbb R_+^d$, and let
\[
A_t(x)=\frac 1{t^d}\int_{(0,t]^d}T_{\mb u}(x)d\mb u, \ \ x\in L^1, \ t>0.
\]
Then
\[
\|A_t(x)\|_1\leq \| x\|_1\text{ \ for all \ }x\in L^1 \text{ \ and \ } \|A_t(x)\|_\ii\leq \| x\|_\ii\text{ \  for all \ }x\in L^1\cap L^\ii.
\]
Therefore, in view of \cite[Proposition 1.1]{cl}, $A_t$ admits a unique extension to a positive Dunford-Schwartz operator acting in $L^1+ \mc M$, which will be also denoted by $A_t$, so we have
\[
A_t(x)=\frac 1{t^d}\int_{(0,t]^d}T_{\mb u}(x)d\mb u,\ \ x\in L^1+\mc M, \ t>0.
\]

Repeating the argument in the proof of Theorem \ref{t13}, we expand Theorem \ref{tflow} to the fully symmetric space
$\mc R_\tau$:

\begin{teo}\label{tflow1}
Let $\{T_{\mb u}:\mb u\in\mbb R_+^d\}\su DS^+$ be a semigroup $L^1$-\,continuous in the interior of $\mbb R_+^d$ on $L^1$. Then for every $x\in\mc R_\tau$ the averages $A_t(x)$  converge b.a.u. as $t\to \ii$ to some $\wh x\in\mc R_\tau$.
\end{teo}
\begin{proof}
By Theorem  \ref{tflow}, the averages $A_{t}(x)$ converge  a.u. whenever $x\in L^1$.
Then, repeating the proof of Theorem \ref{t13} and utilizing the completeness of $L^0$ with respect to a.u. convergence and the coincidence of measure and bilaterally measure topologies \cite[Therem 2.2]{cls}, we conclude that the averages $A_t(x)$ converge a.u. to some $\wh x\in \mc R_\tau$.
\end{proof}

As in Theorem \ref{tsec01}, we obtain the following.

\begin{teo}\label{tflow2}
Let $\tau(\mb 1)=\ii$, let $E=E(\mc M,\tau)$ be a fully symmetric space with $\mb 1\notin E$, and let $\{ T_{\mb u}:\mb u\in\mbb R_+^d\}\su DS^+$ be a semigroup $L^1$-\,continuous in the interior of $\mbb R_+^d$. Then for every $x\in E$, the averages $A_t(x)$ converge  b.a.u. as $t\to \ii$ to some $\wh x\in E$.
\end{teo}

\section{Examples}
Every individual ergodic theorem for $\mc R_\tau$ above is valid for any noncommutative fully
symmetric space $E \su \mc R_\tau$ (with the limit $\wh x\in E$). In this section we give a few examples of
noncommutative fully symmetric subspaces of $\mc R_\tau$.

Since the case $\tau(\mb 1)<\ii$ is trivial in our context, assume that $\tau(\mb 1)=\ii$. As we have noticed
(Proposition \ref{p12}), a symmetric space $E \su L^1+\mc M$ is contained in $\mc R_\tau$ if and only if $\mb 1\notin E$.

1. Let $\Phi$ be an {\it Orlicz function}, that is, $\Phi:[0,\ii)\to [0,\ii)$ is a convex continuous at $0$ function
such that $\Phi(0)=0$ and $\Phi(u)>0$ if $u\ne 0$. Let
\[
L^\Phi=L^\Phi(\mc M, \tau)=\big\{ x \in L^0(\mc M, \tau): \ \tau\big(\Phi(a^{-1}|x|)\big)
<\ii \text{ \ for some \ } a>0\big\}.
\]
be the corresponding {\it noncommutative Orlicz space},  and let
\[
\| x\|_\Phi=\inf\big\{ a>0: \ \tau\big(\Phi(a^{-1}|x|)\big)\leq 1\big\}
\]
be the {\it Luxemburg norm} in $L^\Phi$ (see \cite{cl2}). Since  $\tau(\mathbf 1)=\ii$, we have $\tau\left (\Phi(a^{-1}\mb 1)\right )= \ii$ for all $a>0$, hence $\mb 1  \notin  L^\Phi$.

2. A space $(E,\|\cdot \|_E)$ is said to have {\it order continuous norm} if $\|x_n\|_E\downarrow 0$ whenever $x_n\in E_+$ and $x_n\downarrow 0$.

If $E=E(\mc M,\tau)$  is a noncommutative fully symmetric space  with order continuous norm, then
$\tau\left (\{|x|>\lb\}\right )<\ii$  for all $x\in E$ and $\lb>0$, so $E\su\mc R_\tau$.

3. Let $\varphi$ be an increasing  concave function on $[0,\ii)$ with $\varphi(0)=0$ and $\varphi(t)>0$ for some $t>0$, and let
\[
\Lb_\varphi(\mc M,\tau)=\left\{x \in L^0(\mc M,\tau): \ \|x \|_{\Lb_\varphi}=
\int_0^\ii \mu_t(x) d \varphi(t)<\ii\right\}
\]
be the corresponding {\it noncommutative Lorentz space}; see, for example, \cite{cs}. Since $\Lb_\varphi(0,\ii)$ is a fully symmetric function space \cite[Ch.\,II, \S\,5]{kps} and $\Lb_\varphi(0,\ii)\su\mc R_\mu(0,\ii)$ whenever $\varphi(\ii)=\ii$ \cite[Theorem 1]{mpr}, the fully symmetric noncommutative space $\Lb_\varphi(\mc M,\tau)$ is contained in $\mc R_\tau$.

4. Let $E(0,\ii)$ be a fully symmetric function space. Let $D_s: E(0,\ii) \to E(0,\ii)$, $s>0$, be the bounded linear operator given by $D_s(f)(t) = f(t/s)$, $t > 0$. The {\it Boyd index} $q_E$ is defined as
\[
q_E=\lim\limits_{s \to +0}\frac{\log s}{\log \|D_s\|}.
\]
It is known that $1\leq q_E\leq\ii$ \cite[II, Ch.2, Proposition 2.b.2]{lt}. Since $\|D_s\|\leq\max\{1,s\}$ \cite[II, Ch.\,2.b]{lt}, in the case $\mb 1\in E(\mc M,\tau)$, that is, when $\chi_{(0,\ii)}\in E(0,\ii)$ we have $D_s(\chi_{(0,\ii)})=\chi_{(0,\ii)}$  and $\|D_s\| =1$ for all $s\in (0,1)$, hence $q_E= \ii$ \cite[II, Ch.\,2, Proposition 2.b.1]{lt}. Thus, if $q_E<\ii$, then $\mb 1\notin  E(\mc M)$, and so $E(\mc M,\tau)\su\mc R_\tau$.

\end{document}